\documentclass[letterpaper,11pt]{amsart}

\usepackage[utf8]{inputenc}
\usepackage[english]{babel}
\usepackage[margin = 1.25in]{geometry} 
\usepackage{amsthm}
\usepackage{amsmath}
\usepackage{amssymb}
\usepackage[scr]{rsfso}
\usepackage{csquotes}
\usepackage{bm,bbm}
\usepackage{color}
\usepackage[dvipsnames]{xcolor}
\usepackage{graphicx}
\usepackage{caption, subcaption}
\usepackage{array}
\usepackage{tikz}
\usetikzlibrary{matrix,arrows}
\usepackage{tikz-cd}
\usepackage[framemethod=TikZ]{mdframed}
\usepackage{youngtab}
\usepackage{mathtools}
\usepackage[normalem]{ulem}
\usepackage[all]{xy}
\usepackage{rotating}
\usepackage{extarrows}
\usepackage{accents}
\usepackage{leftidx}
\usepackage{upgreek}
\usepackage{xr-hyper}
\usepackage[draft]{todonotes}
\usepackage{proofread}
\usepackage{hyperref}
\usepackage{macros}

\hypersetup{  
  pdftitle    = {Reconstruction of projective curves from the derived category},
  pdfauthor   = {Dylan Spence},
  colorlinks=true,
  citecolor=orange,
  linkcolor=blue}

\numberwithin{equation}{section}
\title[Reconstruction of projective curves from the derived category]{Reconstruction of projective curves from the derived category}
\begin{document}
\begin{abstract}
    We prove that integral projective curves with Cohen-Macaulay singularities over an algebraically closed field of characteristic zero are determined by their triangulated category of perfect complexes. This partially extends a theorem of Bondal and Orlov to the case of singular projective varieties, and in particular shows that all integral complex projective curves are determined by their category of perfect complexes.
\end{abstract}


\author{Dylan Spence}

\address{
  Department of Mathematics\\
  Indiana University\\
  Rawles Hall\\
  831 East 3rd Street\\
  Bloomington, IN 47405\\
  USA
}
\email{dkspence@indiana.edu}

\maketitle
\setcounter{secnumdepth}{1}
\setcounter{tocdepth}{1}
\tableofcontents

\section{Introduction}

Derived categories of sheaves are a relatively recent development in algebraic geometry. Besides intrinsic interest, they present a fruitful approach to noncommutative and homological methods in algebraic geometry, and since their initial development, many interesting connections to more classical questions have been discovered. For any variety $X$ we can associate to it the bounded derived category of coherent sheaves $\D^b(X)$, which is the Verdier quotient of bounded chain complexes of objects in $\coh X$ by the subcategory of bounded acyclic complexes. The resulting category is no longer abelian, but rather admits a different structure known as a triangulation.

The first question one might ask is how much information about $X$ is lost in this process. It is a well-known theorem of Gabriel that the abelian category $\coh X$ completely determines $X$, but the analogous question for $\D^b(X)$ is more interesting. If we consider $\D^b(X)$ with its tensor product as a symmetric monoidal triangulated category, then it remains a complete invariant \cite{Balmer-Spectrum_of_prime_ideals_in_tensor_categories}. However there are many interesting equivalences that do not respect the tensor product. So instead considering $\D^b(X)$ purely as a triangulated category, we ask the same question. In the case of a smooth projective variety, the answer is tied to the intrinsic geometry of $X$, namely if $\omega_X$ is (anti) ample, then the famous ``reconstruction theorem" of Bondal and Orlov in \cite{BondalOrlov-Reconstruction_from_the_derived_category} states that the derived category is still a complete invariant. However in the case where $\omega_X$ is neither (in particular, trivial), then there are examples of non-isomorphic varieties with equivalent derived categories. Historically the first examples were found by Mukai \cite{Mukai-Duality_and_picard_sheaves}, and the specific functor employed in the equivalence is now referred much more generally as an ``integral functor'', with the term ``Fourier-Mukai functor" typically reserved for those integral functors which are equivalences. Integral functors have taken center stage as the geometric functors of primary consideration, due both to utility and necessity, as a seminal result of Orlov \cite{Orlov-Equivalences_and_K3_surfaces} showed that any derived equivalence between two smooth projective varieties is given by a Fourier-Mukai functor. 

An aspect of the theory which has, up until somewhat recently (see \cite{Ballard-Derived_categories_of_singular_schemes_and_reconstruction,Ruiperez-Fourier-Mukai_transforms_for_Gorenstein_schemes,SalasSalas-Reconstruction_from_the_derived_category,Ruiperez-Relative_integral_functor_for_singular_fibrations_and_partners, Karmazyn-Derived_categories_of_singular_surfaces}), been mostly ignored is the case when $X$ is singular. However many of the basic tools are now in place for its detailed study. For example, the work of Lunts and Orlov \cite{LuntsOrlov-Uniqueness_of_DG_enhancements} has shown that fully faithful functors between projective varieties with arbitrary singularities are integral functors. In addition, many details about the structure of arbitrary triangulated categories are now known. In this paper we are interested in the question of reconstruction, more precisely, to what extent is $X$ determined by $\D^b(X)$ (or $\perf X$) when $X$ is a singular variety. To give some context for our result, it is known that the reconstruction theorem of Bondal and Orlov holds for Gorenstein varieties with ample or anti-ample dualizing bundle, proven by Ballard \cite{Ballard-Derived_categories_of_singular_schemes_and_reconstruction} and de Salas and de Salas \cite{SalasSalas-Reconstruction_from_the_derived_category}. We provide an answer to the reconstruction problem for the next most reasonable type of singularity, those of Cohen-Macaulay type. However due to complications in higher dimensions, we make the additional simplification to consider only $\dim X=1$, namely, projective curves. Since any integral projective curve is automatically Cohen-Macaulay, we lose no generality by our assumptions on the singularities, with arbitrary singularities (necessarily in higher dimensions) left for future work. 

To give more context for out result, let us remind the reader about what is known for dimension one. As a special case of the results cited above, any curve $X$ which is regular or Gorenstein with ample or anti-ample dualizing bundle is completely determined by $\perf X$. The simplicity of curves allows us to reformulate conditions on the canonical bundle to conditions on the arithmetic genus $p_a(X)$. It is well known that any regular curve with $p_a(X) \neq 1$ has anti-ample ($=0$) or ample ($\geq 2$) canonical bundle $\omega_X$. When $X$ is Gorenstein, the situation is only marginally more complicated. As is well-known, $p_a(X)=0$ can only occur when $X \cong \P^1$; on the other hand by Riemann-Roch, the dualizing bundle on a Gorenstein curve with $p_a(X) \geq 2$ has positive degree. Now a theorem of Catanese \cite[Theorem B]{Catanese-Pluricanonical_gorenstein_curves} tells us that the dualizing bundle is ample. 

The only remaining case for regular or Gorenstein curves is when the arithmetic genus is one. Over $\C$, there are few possibilities; the curve is either regular (and hence an elliptic curve) or the curve is Gorenstein, and must be either a cuspidal or nodal cubic. The former case can be handled in several ways, originally in \cite{HilleVandenBergh-FM_transforms} it was shown that any two derived equivalent elliptic curves must be isomorphic (see also \cite{Bernardara-FM_curves}), and the analogous result in the latter cases are proven in \cite{Martin-Fourier-mukai_partners_of_singular_genus_one_curves, MartinPrieto-Derived_equivalences_and_kodaira_fibers}. If we are not over $\C$, see \cite[Section 2.3]{Ward-Arithemtic_Calabi_Yau} for a discussion when $X$ is regular. Assembling the results in this and the previous paragraph, it follows that if $X$ is a regular or Gorenstein complex projective curve, it is completely determined by $\perf X$. 

Our result is the following:
\begin{introthm}{A}\label{Main Theorem}
  Let $X$ and $Y$ be integral projective varieties such that $\dim X =1$ and $X$ is strictly Cohen-Macaulay (i.e. contains a singular point which is not smooth or Gorenstein). Then $\perf X \cong \perf Y$ if and only if $X \cong Y$.
\end{introthm}
\noindent From our result, it then follows that all integral complex projective varieties of dimension one are completely determined by their triangulated category of perfect complexes.

We remark here that since derived equivalences in the singular case are all Fourier-Mukai functors, the work of Ruip\'erez et al. \cite[Theorem 4.4]{Ruiperez-Relative_integral_functor_for_singular_fibrations_and_partners} implies that $\dim Y = \dim X$ and that $Y$ is Cohen-Macaulay. In our proof we will comment on a different way to reach the same conclusion.

Let us now comment on the structure of the paper. We have broken the discussion into two parts, the first part serving as a reminder on some relevant aspects of derived categories. While the original proof of the reconstruction theorem in \cite{BondalOrlov-Reconstruction_from_the_derived_category} works for any equivalence, our proof relies upon, in an essential way, the work of Lunts and Orlov \cite{LuntsOrlov-Uniqueness_of_DG_enhancements} who showed that equivalences between derived categories of projective varieties are given by Fourier-Mukai functors. As such the first subsection is devoted to a review of this very important class of functors. The fundamental result of this subsection is Lemma \ref{lemma_equivalences}, which shows that under our conventions, a Fourier-Mukai equivalence between any of $\D(\qcoh(-))$, $\D^b(-)$, or $\perf (-)$ (see the conventions below for notation) automatically induces a Fourier-Mukai equivalence between the other two pairs. This result is certainly already known to the experts, but we included a proof due to lack of an adequate reference. The following subsection is devoted to a review of a replacement for the Serre functor on the bounded derived category, which following Ballard in \cite{Ballard-Derived_categories_of_singular_schemes_and_reconstruction}, we label a Rouquier functor. 

The only other part of the paper is devoted to the proof. We first take the opportunity to establish a special class of objects of $\perf X$, the perfect zero-cycles. They are a class of skyscraper sheaves on our curve $X$ which are structure sheaves of thickened points; their most useful property is their ability to detect the support of a complex in $\D^b(X)$. After some minor comments about the Rouquier functor in the case of projective curves, the proof then begins in earnest. The proof consists of two major steps. The first is to classify, in a manner very similar to previous work in \cite{BondalOrlov-Reconstruction_from_the_derived_category} and \cite{SalasSalas-Reconstruction_from_the_derived_category}, the perfect zero-cycles on our curves. We then use these objects to show that the Fourier-Mukai functor preserves the structure sheaves of closed points, and well known techniques will finish the proof. 

\subsection{Acknowledgements}
The author wishes to thank Valery Lunts for his valuable input, encouragement, and continual support while advising him on this problem. He also wishes to thank Alexander Polishchuk, Fernando Sancho de Salas, and especially Xuqiang Qin for encouragement and helpful discussions. The author is also indebted to Noah Riggenbach, whose careful reading and criticisms of the initial drafts considerably improved the exposition.  

In addition the author and this project benefitted from the workshops ``Derived Categories, Moduli Spaces and Deformation Theory" and ``The Geometry of Derived Categories" held in Cetraro and Liverpool respectively in 2019, and we thank the organizers and the participants of the events for providing a stimulating research environment. 

\subsection{Conventions}
Throughout we work over a fixed algebraically closed field $k$ of characteristic zero. In addition, unless explicitly stated otherwise, any scheme $X$ is assumed to be a projective variety; in particular for us this means an integral separated scheme of finite type over $\spec k$, such that $X$ admits a closed immersion into some projective space. A (projective) curve refers to a (projective) variety of dimension one. 

Unadorned products of varieties and unadorned tensor products of sheaves on $X$ always mean the fiber product over $\spec k$ and the tensor product over $\O_X$ respectively. We indicate the right derived functor of a left exact functor $F$ by $\R F$, and similarly the left derived functor of a right exact functor $G$ by $\L G$.

By $\D^*(\qcoh X)$ and $\D^*(X)$ we always mean the derived category of complexes of quasi-coherent sheaves and coherent sheaves respectively, with appropriate decoration $* \in \{-,+,b\}$ meaning bounded above, below, and bounded respectively. In particular, we identify $\D^b(X)$ with the full subcategory of $\D(\qcoh X)$ consisting of complexes of quasi-coherent sheaves with bounded coherent cohomology sheaves. By $\perf X$ we mean the full triangulated subcategory of $\D(\qcoh X)$ consisting of complexes which are Zariski-locally isomorphic (in $\D(\qcoh X)$) to a bounded complex of vector bundles. It follows from our assumptions that every perfect complex on $X$ is strictly perfect, that is, globally isomorphic to a bounded complex of vector bundles, as any separated scheme with an ample family of line bundles satisfies this \cite[Theorem 2.4.3]{ThomasonTrobaugh-Higher_K_theory}.
\section{Derived categories of sheaves}

Here we review a few auxiliary facts about derived categories of sheaves which will be of use. The first section is dedicated to the study of integral functors and the focus is mainly on when they can be extended or restricted between subcategories of $\D(\qcoh X)$. The second section focuses on a more general version of the Serre functor on $\D^b(X)$, which was originally studied in \cite{Rouquier-Dimensions_of_triangulated_categories}. 

\subsection{Fourier-Mukai Partners}
Two schemes $X$ and $Y$ are said to be derived equivalent if there is a triangulated equivalence between the two categories $\D(\qcoh X)$ and $\D(\qcoh Y)$. It is a desirable property for any functor $F:\D(\qcoh X) \to \D(\qcoh Y)$ between the two categories to be represented by an object on the product, namely, consider the product: 
\begin{center}
  \begin{tikzcd}
    & X \times Y \arrow[ld, "\pi_X"'] \arrow[rd, "\pi_Y"] &   \\
  X &                                                     & Y,
  \end{tikzcd}
\end{center}
with $\pi_X$ and $\pi_Y$ the projections. Given some object $\K$ in $\D(\qcoh (X \times Y))$, called the kernel, we ask if $F$ is isomorphic to the triangulated functor $$\Phi^{X \to Y}_\K (-) = \R \pi_{Y*} (\pi_X^*(-) \tens{\L} \K):\D(\qcoh X) \to \D(\qcoh Y).$$ If so, we say that $F$ is an integral functor with kernel $\K$. If it is also an equivalence, we say that it is a Fourier-Mukai functor. 

Fourier-Mukai functors also make sense when, instead of considering the unbounded derived category of quasi-coherent sheaves, we consider other suitable subcategories, namely $\perf X$ and $\D^b(X)$. For both of these subcategories, we make a nearly identical definition as above. 

It is a theorem of Orlov \cite[Theorem 2.2]{Orlov-Equivalences_and_K3_surfaces} that for $X$ and $Y$ smooth projective varieties, then any triangulated fully faithful functor admitting adjoints between $\D^b(X)$ and $\D^b(Y)$ is an integral functor. The question then is to extend this to the singular case, and using the formalism of dg-enhancements, Lunts and Orlov were able to show the following.

\begin{lemma}[\cite{LuntsOrlov-Uniqueness_of_DG_enhancements}, Corollary 9.12]\label{lemma_FM for qcoh}
  Let $X$ and $Y$ be quasi-compact separated schemes over a field $k$. Assume that $X$ has enough locally free sheaves. Let $F : \D(\qcoh X) \to \D(\qcoh Y)$ be a fully faithful functor that commutes with coproducts. Then there is an object $\K \in \D(\qcoh (X \times Y))$ such that the functor $\Phi^{X \to Y}_\K$ is fully faithful and $\Phi^{X \to Y}_\K(C) \cong F(C)$ for any $C \in \D(\qcoh X)$. 
\end{lemma} 
\noindent We emphasize that it may \emph{not} be true that there is an isomorphism of functors $F \cong \Phi^{X \to Y}_\K$. From elementary arguments however, it is true that if $F$ is an equivalence, then $\Phi^{X \to Y}_\K$ will be as well. 

However the main object of study, for many technical reasons, is typically not $\D(\qcoh X)$. Instead one works with the better behaved subcategories $\D^b(X)$ and $\perf X$. Again, Lunts and Orlov were able to give a very general result, dependent on the geometry of $X$.

\begin{lemma}[\cite{LuntsOrlov-Uniqueness_of_DG_enhancements}, Corollary 9.17]\label{LuntsOrlov_FM for coh}
  Let $X$ be a projective scheme with maximal torsion subsheaf of dimension zero $T_0(\O_X)$ trivial, and $Y$ a noetherian scheme. Let $F:\D^b(X) \to \D^b(Y)$ be a fully faithful functor that has a right adjoint. Then there is an object $\K \in \D^b(X \times Y)$ such that $F \cong \Phi^{X \to Y}_\K|_{\D^b(X)}$.
\end{lemma}

For the category of perfect complexes, they were also able to conclude:

\begin{lemma}[\cite{LuntsOrlov-Uniqueness_of_DG_enhancements}, Corollary 9.13]\label{LuntsOrlov_FM for perf}
  Let $X$ be a projective scheme with maximal torsion subsheaf of dimension zero $T_0(\O_X)$ trivial, and $Y$ a quasi-compact and separated scheme. If $F: \perf X \to \perf Y$ is a fully faithful functor, then there is an object $\K$ in $\D^b(X \times Y)$ such that $F \cong \Phi^{X \to Y}_\K|_{\perf X}$.
\end{lemma}

\begin{remark}
  In Theorem \ref{Main Theorem}, the assumption that $X$ was integral allows us to conclude that $\O_X$ is a torsion-free sheaf, and so the hypotheses of Lemmas \ref{LuntsOrlov_FM for coh} and \ref{LuntsOrlov_FM for perf} are satisfied.
\end{remark}

\begin{remark}\label{remark_FM coproducts}
  We record here for the reader that integral functors in great generality will commute with all coproducts. The only difficulty is the derived pushforward, but so long as $f:X \to Y$ is quasi-compact and quasi-separated, $\R f_*$ will commute with arbitrary coproducts, see \cite[\href{https://stacks.math.columbia.edu/tag/08DZ}{Lemma 08DZ}]{stacks-project}. Alternatively one may use the existence of adjoints \cite{Rizzardo-Adjoints_for_FM}.
\end{remark}

\begin{remark}\label{remark_perf}
  It is worthwhile to point out that working with $\D^b(X)$ can be troublesome in the singular case. If the kernel of an integral functor is not nice enough the derived tensor product can escape to the bounded above derived category. For some criteria for an object to define a functor between bounded derived categories, see \cite[Proposition 2.7]{Ruiperez-Relative_integral_functor_for_singular_fibrations_and_partners}. Alternatively, one could work directly with $\D^-(X)$.
\end{remark}

Although in our assumptions in Theorem \ref{Main Theorem} are $\perf X \cong \perf Y$, it so happens that we will need to pass to $\D^b(X)$ eventually. This is a delicate question since we are in the singular setting. Before we providing an answer we should first recall the notion of a compact object in a triangulated category and the notion of a compactly generated category. 

\begin{definition}\label{def_compact and compact generation}
 Given a triangulated category $\T$, an object $A$ is called compact if the functor $\hom_\T(A,-)$ commutes with all coproducts. The full thick subcategory of compact objects is denoted $\T^c$. $\T$ is said to be compactly generated if $(\T^c)^\perp =0$.
\end{definition}

\noindent It is well known (see \cite[Theorem 3.1.1]{BondalVandenBergh-Generators_and_representability}) that if $X$ is quasi-compact and quasi-separated, $\D(\qcoh X)$ is a compactly generated triangulated category, generated by a single perfect complex, and $\D(\qcoh X)^c = \perf X$. Compact generation also provides a very convenient characterization of the category $\T$ in question. The following is due to Neeman in \cite[Lemma 3.1]{Neeman-Grothendieck_duality_via_Brown_representability}. 

\begin{lemma}\label{lemma_coproducts and compact generation}
 Let $\mathcal{S} \subset \T$ be a full triangulated subcategory of a compactly generated triangulated category $\T$. If $\mathcal{S}$ contains $\T^c$ and is closed under all coproducts of its objects, then $\mathcal{S} \cong \T$.
\end{lemma}
\noindent Our primary example is $\T^c = \perf X$ and $\T = \D(\qcoh X)$. In particular, any full subcategory of $\D(\qcoh X)$ which contains $\perf X$ and all coproducts of objects in $\perf X$ must be everything. 

\begin{lemma}\label{lemma_equivalences}
  Let $X$ and $Y$ be two projective varieties over the field $k$. Then given $\K \in \D(\qcoh (X \times Y))$, the following are equivalent:
  \begin{enumerate}
    \item $\Phi_\K^{X \to Y}: \D(\qcoh X) \overset{\sim}{\to} \D(\qcoh Y)$,
    \item $\Phi_\K^{X \to Y}: \D^b(X) \overset{\sim}{\to} \D^b(Y)$,
    \item $\Phi_\K^{X \to Y}: \perf X \overset{\sim}{\to} \perf Y$,
  \end{enumerate}
  Furthermore, if any one of the three above occurs, then $\K\in \D^b(X \times Y)$.
\end{lemma}

One should also compare this result to \cite[Proposition 7.4]{CanonacoStellari-Uniqueness_of_DG_enhancements}.
\begin{proof}
  We first show that (1) implies (3). Indeed, $\perf X \subset \D(\qcoh X)$ are precisely the compact objects, namely those objects $A$ such that $\hom(A,-)$ commutes with all coproducts. It is then immediate that $\Phi^{X\to Y}_\K$ restricts to an equivalence $\perf X \overset{\sim}{\to} \perf Y$. This shows $(1) \implies (3)$. 
  
  Again assume that (1) holds. By \cite[Proposition 7.46]{Rouquier-Dimensions_of_triangulated_categories}, an object $B\in \D(\qcoh X)$ is in $\D^b(X)$ if and only if for all $A \in \perf X$ we have $$\dim \oplus_i \hom(A,B[i]) <\infty.$$ Since, by the previous paragraph, the integral functor preserves compact objects, applying this fact shows that $\Phi^{X \to Y}_\K$ restricts to an equivalence $$\Phi^{X \to Y}_\K|_{\D^b(X)}:\D^b(X) \overset{\sim}{\to} \D^b(Y).$$ This shows that $(1) \implies (2)$.

  Let us now assume (2). By \cite[Lemma 7.49]{Rouquier-Dimensions_of_triangulated_categories}, an object $A \in \D(\qcoh X)$ belongs to $\perf X$ if and only if $$\dim \oplus_i \hom(A,B[i]) < \infty$$ for all $B \in \D^b(X)$. The case when $A \in \D^b(X)$ implies that $\Phi^{X \to Y}_\K$ restricts to an equivalence $$\Phi_\K^{X \to Y}|_{\perf X}: \perf X \overset{\sim}{\to} \perf Y.$$ This shows $(2) \implies (3)$.

  To complete the proof of the first claim, it is enough to show that (3) implies (1). Consider the naive extension of this functor to $\Phi^{X \to Y}_\K:\D(\qcoh X) \to \D(\qcoh Y)$ by simply allowing the Fourier-Mukai functor to take inputs from $\D(\qcoh X)$. We claim that this functor is fully faithful on all of $\D(\qcoh X)$. Let $\langle \perf X \rangle^{\oplus} \subset \D(\qcoh X)$ be the minimal full triangulated subcategory containing $\perf X$ and all coproducts of perfect objects. By Lemma \ref{lemma_coproducts and compact generation}, $\langle \perf (X) \rangle^{\oplus} = \D(\qcoh X)$. 
  Consider the full triangulated subcategory of $\D(\qcoh X)$ defined by $$T = \{ Q \, | \, \Phi_\K^{X \to Y} : \hom(E,Q) \cong \hom (\Phi^{X \to Y}_\K(E),\Phi^{X \to Y}_\K(Q)), \text{ for all } E \in \perf X \}.$$ By assumption, this category contains $\perf X$, and since $E$ is compact and Fourier-Mukai functors commute with all coproducts, it also contains $\langle \perf X \rangle^{\oplus}$. Thus $T = \D(\qcoh X)$. Similarly, consider the full triangulated subcategory of $\D(\qcoh X)$ defined by $$S = \{Q \, | \, \Phi_\K^{X \to Y} : \hom(Q,F) \cong \hom (\Phi^{X \to Y}_\K(Q),\Phi^{X \to Y}_\K(F)), \text{ for all } F \in \D(\qcoh X) \}.$$ Since we have shown that $T = \D(\qcoh X)$, the same logic again implies $\langle \perf X \rangle^{\oplus} \subset S$, and so must be all of $\D(\qcoh X)$. This shows that the extension to the unbounded derived category is fully faithful.

  Now to see that the extension is essentially surjective, it again is enough by Lemma \ref{lemma_coproducts and compact generation} that the essential image be a full triangulated subcategory and is closed under coproducts of objects in $\perf Y$. The essential image is full as we have shown above that the functor is fully faithful, and the second condition follows as the Fourier-Mukai functor preserves all coproducts. This shows that (3) implies (1), and completes the proof of the first claim. Let us now turn to the second statement. 

  By restricting to $\perf X$, by the work of Lunts and Orlov in Lemma \ref{LuntsOrlov_FM for perf}, we conclude that there is an object $\E \in \D^b(X \times Y)$ such that $\Phi^{X \to Y}_\E \cong \Phi^{X \to Y}_\K$. Now by \cite[Remark 5.7]{CanonacoStellari-FM_supported}, the kernel is unique, and so $\E \cong \K$. Since this kernel realizes all three equivalences, the result follows.
\end{proof}

To conclude this subsection, we recall a well-known fact, namely if an integral functor between two projective varieties preserves structure sheaves of points, then the kernel of the functor is supported on the graph of a morphism between the varieties in question. In particular, if the integral functor is a Fourier-Mukai functor, the two varieties are isomorphic. This was originally known for smooth projective varieties (see for example, \cite[Corollary 5.23 and 6.14]{Huybrechts-FM_transforms_in_algebraic_geometry}), but the crux of the proof is the knowledge that the functor is represented by an object on the product. Since this is known for singular varieties, the same statement holds. 

\begin{lemma}\label{lemma_skyscraper sheaves give iso}
  Let $\Phi^{X \to Y}_\K:\D^b(X) \to \D^b(Y)$ be an integral functor between projective varieties such that for all closed points $x \in X$, $\Phi^{X \to Y}_\K(k(x)) \cong k(y)[r]$ for some closed point $y \in Y$ and $r\in \Z$. Then $\K \cong \sigma_* \mathcal{L} [m]$ for a section $\sigma:X \to X \times Y$ of $\pi_X$, $\mathcal{L} \in \operatorname{Pic} X$, and $m\in \Z$, and $$\Phi^{X \to Y}_\K(-) \cong \R f_* ( (-) \tens{} \mathcal{L})[m]$$ for a morphism $f:X \to Y$. Further, $f$ is an isomorphism if and only if $\Phi^{X \to Y}_\K$ is an equivalence of categories. 
\end{lemma}
Recall that the support of a complex $\K \in \D^b(X)$, denoted $\operatorname{supp} \K$, is the union of the supports of its cohomology sheaves. In this scenario the support is a closed subscheme.
\begin{proof}
  The proof of the first statement can be found in several places, for example, \cite[Proposition 4.2]{HilleVandenBergh-FM_transforms}, \cite[Corollary 1.12]{Bartocci-FM_in_geometry_and_physics}, or originally, in steps 2 and 3 of the proof of \cite[Theorem 2.10]{Orlov-Equivalences_of_abelian_varieties} where it is stated in the setting of abelian varieties. We note here that as a byproduct of the proof, the section $\sigma:X \to X \times Y$ is an isomorphism between $X$ and $\operatorname{supp} \K$. 

  By letting $f= \pi_Y \circ \sigma$, the projection formula easily shows that $$\Phi^{X \to Y}_\K (-) \cong \R f_* ( (-) \tens{} \mathcal{L})[m],$$ and it is clear that this functor is an equivalence if and only if $f$ is an isomorphism. Indeed, the ``if'' direction is obvious, for the converse, the inverse functor $\Phi^{Y \to X}_\E$ must satisfy the same hypotheses, and hence there is a section $\zeta:Y \to X \times Y$ realizing an isomorphism between $Y$ and $\operatorname{supp} \K$. Setting $g = \pi_X \circ \zeta$, it follows that $g$ is the inverse to $f$. 
\end{proof}

\begin{remark}
  There are other ways to see that an equivalence satisfying $\Phi^{X \to Y}_\K(k(x)) \cong k(y)[r]$ implies that $X \cong Y$. For example, the reader should see \cite[Lemma 3.5]{Favero-Reconstruction_and_finiteness} for a proof which is closer in spirit to the original proof of Bondal and Orlov in \cite{BondalOrlov-Reconstruction_from_the_derived_category}.
\end{remark}

\subsection{Rouquier Functors}
All results in this section are contained in \cite[Section 4]{Rouquier-Dimensions_of_triangulated_categories} and \cite[Section 5]{Ballard-Derived_categories_of_singular_schemes_and_reconstruction}. 

Given a smooth projective variety $X$, the bounded derived category $\D^b(X)$ possesses a distinguished autoequivalence given by $$S_X(-) = (-) \tens{} \omega_X[\dim X] : \D^b(X) \to \D^b(X).$$ This functor captures the notion of Serre duality, and as such is called a Serre functor. More generally given a $k$-linear triangulated category $\T$, a Serre functor is a triangulated autoequivalence $S:\T \to \T$ for which there are natural isomorphisms $$\eta_{A,B}:\hom_{\T} (B,S(A)) \to \hom_{\T}(A,B)^*,$$ where ``$*$" indicates the vector space dual. Typically the presence and properties of a Serre functor on a triangulated category $\T$ has significant consequences for the structure of $\T$. 

The first question is that of existence. Fixing $A \in \T$, we see that if the functor $\hom_\T(A,-)^*$ is a representable functor for any such $A$, with representing object in $\T$, then $\T$ has a so-called weak Serre functor, that is, a Serre functor that is not necessarily an autoequivalence. The question of representability can be difficult, but if $\T$ is compactly generated and cocomplete, a convenient criterion exists known as Brown representability:

\begin{theorem}[\cite{Neeman-Grothendieck_duality_via_Brown_representability}, Theorem 3.1]
  Let $\T$ be a $k$-linear, compactly generated, and cocomplete triangulated category. If $H:\T \to \text{Vect}_k$ is a contravariant functor sending all distinguished triangles to long exact sequences, and $H$ sends all coproducts to products, then $H$ is representable.
\end{theorem}

\noindent It is well-known that the functors $\hom_\T(A,-)$ are always cohomological, that is, they satisfy the hypotheses of Brown representability; so to conclude that $\T$ has a weak Serre functor by this approach, it follows that the functor $\hom_\T(A,-)$ must commute with all coproducts. This is the same as saying that $A$ is compact in $\T$, and so it follows that if this functor is represented by an object in $\T^c$, the category $\T^c$ will have a weak Serre functor.  

Let us now specialize to the geometric case, where $\T = \D(\qcoh X)$. It is well-known that $X$ is regular if and only if $\perf X \cong \D^b(X)$, and so in the singular case we no longer expect $\D^b(X)$ to possess a Serre functor. However it is still possible for $\perf X$; if $X$ is Gorenstein, this is precisely what occurs.

\begin{proposition}[\cite{Ballard-Derived_categories_of_singular_schemes_and_reconstruction, SalasSalas-Reconstruction_from_the_derived_category}]
  If $X$ is a proper variety over a field $k$, the category $\perf X$ has a Serre functor if and only if $X$ is Gorenstein. In particular, the functor has the formula: $S_X(-) = (-) \tens{} \omega_X [\dim X]$.
\end{proposition}

The failure of $\perf X$ to have a Serre functor in the non-Gorenstein case is not due to a failure of existence, but rather the representing objects for the functors $\hom(A,-)$ no longer lie in $\perf X$. This observation is due to Rouquier \cite[Section 4]{Rouquier-Dimensions_of_triangulated_categories}, where the following definition was given. 
\begin{definition}
  Let $\mathcal{S}$ and $\T$ be arbitrary $k$-linear triangulated categories, and $F:\mathcal{S} \to \T$ be a $k$-linear triangulated functor. A Rouquier functor associated to $F$ is a $k$-linear triangulated functor $R_F : \mathcal{S} \to \T$ for which there are natural isomorphisms $$\eta_{A,B}: \hom_{\T} (B,R_F(A)) \to \hom_{\T}(F(A),B)^*.$$
\end{definition}

In particular, if $\mathcal{S} = \T$ and $F$ is the identity functor, then a Rouquier functor is just a weak Serre functor. As we had discussed above, the existence of a Serre functor is closely related to the representability of the functors $\hom_\T(A,-)^*$, and in the case of a Rouquier functor, a similar result also holds.

\begin{lemma}[\cite{Ballard-Derived_categories_of_singular_schemes_and_reconstruction}, Lemma 5.7]
  A necessary and sufficient condition for the existence of $R_F$ is the representability of the functors $\hom_{\T}(F(A),-)^*$ for $A$ an object of $\mathcal{S}$. If it exists, its unique. 
\end{lemma}

\noindent In particular, if the triangulated category is compactly generated, then the previous discussion implies the following.

\begin{lemma}[\cite{Rouquier-Dimensions_of_triangulated_categories}, Corollary 4.23]
  Let $\T$ be a $k$-linear, compactly generated, and cocomplete triangulated category. Then the Rouquier functor associated to the inclusion $\T^c \to \T$ exists. Furthermore, if $\T^c$ has finite dimensional Hom spaces, then the Rouquier functor is faithful.
\end{lemma}

Our main example is the case of projective schemes. For the following, see Proposition 7.47 and Remark 7.48 in \cite{Rouquier-Dimensions_of_triangulated_categories}, see also \cite[Example 5.12]{Ballard-Derived_categories_of_singular_schemes_and_reconstruction} for more details.

\begin{lemma}\label{lemma_Rouquier functor is dualizing complex}
  Let $f:X \to \spec k$ be a projective scheme over the field and $\iota:\perf X \hookrightarrow D(\qcoh X)$ the inclusion functor. Then the Rouquier functor $R_X := R_\iota$ exists and is isomorphic to $$(-) \tens{\L} f^! \O_{\spec k}.$$ Furthermore, it maps $\perf X$ into $\D^b(X)$.
\end{lemma}

Another extremely useful feature of a Serre functor is that it commutes with all equivalences of triangulated categories. In our situation, there is an analogous statement.

\begin{lemma}[\cite{Ballard-Derived_categories_of_singular_schemes_and_reconstruction}, Lemma 5.15]\label{lemma_Rouquier commutes with equivalences}
  Assume that there are equivalences of triangulated categories $\phi: \T \to \T'$ and $\psi:\mathcal{S} \to \mathcal{S}'$ and functors $F:\mathcal{S} \to \T$ and $F':\mathcal{S}' \to \T'$ making the diagram
  \begin{center}
    \begin{tikzcd}
      \mathcal{S} \arrow[r,"\psi"] \arrow[d,"F"] & \mathcal{S}' \arrow[d,"F'"] \\
      \T \arrow[r,"\phi"] & \T'
    \end{tikzcd}
  \end{center}
  commute. Then if $R_F$ exists, so does $R_{F'}$, and moreover $R_{F'} \circ \psi \cong \phi \circ R_F$. 
\end{lemma}

\section{Reconstruction for projective curves}

In this section, after some more preliminaries, we give the proof of Theorem \ref{Main Theorem}. We first establish a class of objects on projective curves which we call perfect zero-cycles; one may think of these as ``test objects" against which we will find the support of particular complexes. Following this, we exploit the perfect zero-cycles to show that the Fourier-Mukai transform which realizes the equivalence $\perf X \cong \perf Y$ preserves the structure sheaves of closed points. Then the theorem will follow from Lemma \ref{lemma_skyscraper sheaves give iso}. 

\subsection{Preliminaries: Perfect zero-cycles}
In the original proof, due to Bondal and Orlov in \cite{BondalOrlov-Reconstruction_from_the_derived_category}, point objects in $\D^b(X)$ were characterized as the structure sheaves of closed points by using the Serre functor. In the singular case, since structure sheaves of closed point are no longer perfect, we work primarily with $\perf X$ for the reasons mentioned in Remark \ref{remark_perf}.

Let us first recall from \cite[Section 1.3]{Ruiperez-Fourier-Mukai_transforms_for_Gorenstein_schemes} a construction of what we will refer to as \emph{Koszul zero-cycles}. Let $X$ be a projective curve (which by our conventions, is integral). Then for any closed point $x \in X$, the local ring $(\O_{X,x},\mathfrak{m}_x)$ is an integral noetherian ring of dimension one. Choosing any element of the maximal ideal yields a non-zero divisor $f \in \mathfrak{m}_x \subset \O_{X,x}$, and letting $(f)$ be the ideal sheaf determined locally by this nonzero divisor, this determines a closed subscheme $Z_x$, supported only at $x\in X$. Since $f$ is a nonzero divisor, multiplication by a local section yields a resolution of the structure sheaf $\O_{Z_x} = \O_X/(f)$ by locally free sheaves of length 1 ($=\dim X$), hence $\O_{Z_x}$ is perfect as an object of $\D^b(X)$. It is worthwhile to point out that this is just the Koszul complex associated to a regular sequence of length one, which provides some justification for the name.

\begin{remark}\label{remark_reduced dim one}
  Much more generally, a sequence of elements $(x_1,...,x_n) \subset \mathfrak{m}$ in a noetherian local ring $R$ is called a regular sequence if $x_i$ is a nonzero divisor on $R/(x_1,...,x_{i-1})$ for all $i \leq n$ and $R/(x_1,...,x_n) \neq 0$. The length of any maximal regular sequence is constant and is called the depth of $R$, denoted by $\operatorname{depth}R$. The depth should be thought of as an algebraic definition of the notion of dimension; in general the depth and dimension do not agree, but they can be related via $\operatorname{depth} R \leq \dim R$. Further, if $I$ is an ideal such that $R/I$ has finite projective dimension as an $R$-module, there is a close relationship between the depth of a ring and the projective dimension of $R/I$ (more generally, any $R$-module of finite projective dimension), $$\operatorname{proj.dim.} R/I + \operatorname{depth} R/I = \operatorname{depth} R,$$ known as the Auslander-Buchsbaum formula \cite[Therem 1.3.3]{BrunsHerzog-CM_rings}.  Finally, a ring is said to be a Cohen-Macaulay ring if $\operatorname{depth} R = \dim R$. 
  
  We can also globalize this, namely a locally noetherian scheme $X$ is said to be a Cohen-Macaulay scheme if $\O_{X,x}$ is a Cohen-Macaulay ring for all $x\in X$. While we will only need the most basic facts about Cohen-Macaulay rings, a more thorough treatment can be found in \cite{BrunsHerzog-CM_rings}.
  
  Our primary example is when $R$ is a reduced local ring of dimension one, where such rings are automatically Cohen-Macaulay (this is \cite[Exercise 2.1.20(a)]{BrunsHerzog-CM_rings}). Since the definition of a Cohen-Macaulay scheme is local, it follows that any integral curve is Cohen-Macaulay. A maximal regular sequence in such a ring $R$ is simply a nonzero divisor $f \in R$ as we found above, and it follows that the quotient satisfies $\dim R/(f) = 0$. As indicated, we also have a finite resolution by free $R$-modules of finite rank given by the Koszul resolution: $$0 \to R \overset{\cdot f}{\to} R \to R/I \to 0.$$ 
\end{remark}

\begin{remark}\label{remark_koszul zero cycle}
  Considering again the geometric situation where $X$ is a projective curve, for any $P\in \perf X$, $\operatorname{End}_{\perf X} (P)$ is a finite dimensional (noncommutative) $k$-algebra with multiplication given by composition. Now for any other object $Q \in \perf X$, the space $\hom_{\perf X}(P,Q)$ has the structure of a right $\operatorname{End}_{\perf X}(P)$-module, with the action given by pre-composition. Choosing some closed point $x\in X$, if we consider a Koszul zero-cycle $P = \O_{Z_x}$, since $Z_x$ is affine, we can identify $\O_{Z_x}$ with its global sections and thus $$\hom_{\O_X}(\O_{Z_x},\O_{Z_x}) \cong \hom_{\perf X}(\O_{Z_x},\O_{Z_x}) \cong \O_{Z_x},$$ which shows that in this case the endomorphism algebra is commutative. Further, since $\ext$ is local, the Koszul resolution yields $$\hom_{\perf X}(\O_{Z_x},\O_{Z_x}[1]) \cong \O_{Z_x}.$$ Thus the right $\O_{Z_x}$-module $\hom_{\perf X} (\O_{Z_x},\O_{Z_x}[1])$ is a cyclic module. 
\end{remark}

These Koszul zero-cycles will be the prototypical example of our point-like objects. More generally however, there is, to the author's knowledge, no good categorical way to test if an ideal is generated by a regular sequence. So we give below the definition of a perfect zero-cycle, whose definition was inspired by \cite[Definition 1.6]{SalasSalas-Reconstruction_from_the_derived_category} in the context of Gorenstein schemes.

\begin{definition}
  Let $X$ be projective curve and let $x \in X$ be a closed point. A perfect zero-cycle at $x$ is a closed subscheme $Z_x \hookrightarrow X$ supported at $x$ such that
  \begin{enumerate}
    \item $Z_x$ is a zero dimensional scheme,
    \item $\O_{Z_x}$ is an object of $\perf X$, and
    \item $\hom_{\perf X} (\O_{Z_x},\O_{Z_x}[1])$ is a cyclic right $\O_{Z_x}$-module.
  \end{enumerate} 
\end{definition}

Here and from now on, we will always denote a perfect zero-cycle supported at $x\in X$ by $Z_x$.

\begin{remark}\label{remark_enough zero cycles}
  By Remark \ref{remark_koszul zero cycle}, it is clear that every Koszul zero-cycle is a perfect zero-cycle, and since we can construct a Koszul zero-cycle for every closed point $x \in X$, this shows that the category $\perf X$ has ``enough perfect zero-cycles''. 
\end{remark} 

\begin{remark}\label{remark_regular iff perfect}
  Clearly any regular closed point $x\in X$ with its structure sheaf $k(x)$ is a Koszul zero-cycle. Conversely, if a closed point $x\in X$ with structure sheaf $k(x)$ is a perfect zero-cycle, it follows by assumption that $k(x)$ is perfect as an object of $\D^b(X)$. The well-known result of Auslander, Buchsbaum, and Serre (see for example, \cite[Theorem 2.2.7]{BrunsHerzog-CM_rings}), which states that a noetherian local ring $R$ is regular if and only if its residue field has finite projective dimension as an $R$-module. Thus $x\in X$ must be regular.
\end{remark}

One of the most important uses of the set of perfect zero-cycles is their ability to probe the support of a complex. Recall that the support of a complex $Q \in \D^b(X)$, denoted $\operatorname{supp} Q$, is the union of the supports of its cohomology sheaves. Since there are finitely many non-zero cohomology sheaves, and each cohomology sheaf is coherent, this gives a closed subscheme of $X$. A priori this subscheme could be non-reduced, but we will always implicitly take the reduced subscheme structure when we refer to the support of a complex. 

For the reader's convenience, we include the following elementary lemma.

\begin{lemma}\label{lemma_local maps}
  Let $(R,\mathfrak{m},k)$ be a local noetherian ring and $M$ a finitely generated module. If $\mathfrak{m} \in \operatorname{supp} M$, $M$ admits a surjection $M \to k$. Further if $\{\mathfrak{m}\} = \operatorname{supp}M$, then $M$ also admits an injection $k \to M$.
\end{lemma}
\begin{proof}
  If $M$ is finitely generated and supported at the maximal ideal, then Nakayama's lemma implies that $M/\mathfrak{m}M \neq 0$, and further is a finite dimensional $k$-vector space. Choosing any further projection to a one-dimensional subspace, we have a surjection $M \to k$. 

  Consider a prime ideal $\mathfrak{p} \subset R$. Recall that $\mathfrak{p}$ is said to be an associated to $M$ if and only if $R/\mathfrak{p}$ is isomorphic to a submodule of $M$. Further, since $R$ is noetherian, standard results in commutative algebra imply that the set of associated primes is always nonempty and is a subset of the support of $M$. Now if $\{\mathfrak{m}\} = \operatorname{supp}M$, it follows that $\mathfrak{m}$ is an associated prime to $M$, and hence there is a submodule of $M$ which is isomorphic to $k$. The inclusion of submodules $k \to M$ yields the claim.
\end{proof}

The reader should compare the following result to \cite[Proposition A.91]{Bartocci-FM_in_geometry_and_physics} and \cite[Lemma 3.5]{Ruiperez-Relative_integral_functor_for_singular_fibrations_and_partners}.

\begin{lemma}\label{lemma_perfect point test support}
  Let $X$ be a Cohen-Macaulay variety and $Q \in \D^b(X)$. Given a closed point $x\in X$, $x \in \operatorname{supp}Q$ if and only if $\hom_{\D^b(X)}(Q,\O_{Z_x}[m]) \neq 0$ for some $m \in \Z$ and all perfect zero-cycles $Z_x$ supported at $x$.
\end{lemma}
We note here that we only need the property that a the structure sheaf of a perfect zero-cycle is a supported at only a single closed point. 
\begin{proof}
  Since $X$ is assumed Cohen-Macaulay, for any closed point there is a perfect zero-cycle supported on it (see Remark \ref{remark_enough zero cycles}). Suppose a closed point $x\in X$ belongs to the support of $Q$ and denote the cohomology sheaves of $Q$ by $\mathcal{H}^i$. We consider the spectral sequence
  \begin{equation}\label{eqn_hyperext1}
    E^{p,q}_2 = \hom_{\D^b(X)}(\mathcal{H}^{-q},\O_{Z_{x}}[p]) \implies \hom_{\D^b(X)} (Q , \O_{Z_x}[p+q]).
  \end{equation}
  Since $x$ belongs to the support of $Q$, there is at least one cohomology sheaf $\mathcal{H}^{q}$ supported at $x$. Let $\iota_x:\{x\} \hookrightarrow X$ be the inclusion of the closed point $x$, and using the adjunction between the inverse image and pushforward, there is an isomorphism: $$\hom_{\O_{X,x}}(\mathcal{H}^{q}_{x},\O_{Z_{x}}) \cong \hom_{\O_X}(\mathcal{H}^{q},\O_{Z_{x}}),$$ where we identify $\O_{Z_x}$ with its pushforward since it is a skyscraper sheaf on $X$.
  
  To show that these spaces are nonzero it is clearly enough to see that the left-hand side is nonzero.  In particular the stalk $\mathcal{H}^q_x$ is a finitely generated $\O_{X,x}$-module which is supported on the maximal ideal $\mathfrak{m}_x$ by assumption, and $\O_{Z_x}$ is a finitely generated $\O_{X,x}$ module whose support is precisely the maximal ideal. Hence by Lemma \ref{lemma_local maps} there is a nonzero morphism $$\mathcal{H}^q_x \to k(x) \to \O_{Z_x},$$ which shows that the morphism space is nonzero. 
  
  So now taking the maximal $q_0$ such that $\mathcal{H}^{q_0}_x \neq 0$ at $x$, we see that the term $E^{0,-q_0}_2$ survives to the $E_\infty$ page as there are no negative Ext groups between two sheaves, so the differential ending at $E^{0,-q_0}_2$ starts from zero, and the differential leaving $E^{0,-q_0}_2$ hits terms which are zero simply because we had chosen $q_0$ maximal, and all Ext groups between sheaves with disjoint support are trivial  (using the local-to-global spectral sequence). Hence this term survives, and since it is nonzero, we see $\hom_{\D^b(X)}(Q,\O_{Z_x}[-q_0]) \neq 0$.

  Conversely, suppose that $\hom_{\D^b(X)}(Q,\O_{Z_x}[m]) \neq 0$ for some $m \in \Z$. Suppose by way of contradiction that $x \notin \operatorname{supp} Q$, equivalently, $x \notin \operatorname{supp}\mathcal{H}^i$ for all $i \in \Z$. But since sheaves with disjoint support have no nontrivial Ext groups, the spectral sequence \ref{eqn_hyperext1} yields a contradiction.
\end{proof}

\subsection{Preliminaries: The Rouquier functor}
Before we move on to the proof, there is one more piece of the puzzle which we should elaborate on. In Lemma \ref{lemma_Rouquier functor is dualizing complex} we recalled from \cite[Remark 7.48]{Rouquier-Dimensions_of_triangulated_categories} that the Rouquier functor $R_X$ associated to any projective scheme is given by $$R_X(-) = (-) \tens{\L} f^! \O_{\spec k},$$ and has essential image in $\D^b(X)$. We would now like to take the opportunity to remind the reader of the relevant features of the dualizing complex $f^! \O_{\spec k}$ in the setting where $X$ is a projective curve. By assumption we have a closed immersion $i:X \hookrightarrow \P^n$, and so the dualizing complex may be taken to be \cite[\href{https://stacks.math.columbia.edu/tag/0AA2}{Lemma 0AA2}, \href{https://stacks.math.columbia.edu/tag/0AA3}{Lemma 0AA3}]{stacks-project} $$f^! \O_{\spec k}  = \R\shom(i_*\O_X, \omega_{\P^n}[n])|_X.$$

Since $X$ is Cohen-Macaulay by Remark \ref{remark_reduced dim one}, this complex has exactly one non-zero cohomology sheaf in degree $-1$ \cite[\href{https://stacks.math.columbia.edu/tag/0BS2}{Lemma 0BS2}]{stacks-project}, and so the complex is quasi-isomorphic to its unique cohomology sheaf $$f^! \O_{\spec k} \cong \mathcal{H}^{-1} (f^! \O_{\spec k})[1].$$ We denote this cohomology sheaf by $\omega_X$. This is nothing more then the classical dualizing sheaf appearing in Serre duality  \cite[\href{https://stacks.math.columbia.edu/tag/0BS3}{Lemma 0BS3}]{stacks-project}. 
Putting this all together, if $X$ is a projective curve, the Rouquier functor can be given explicitly as $$R_X(-) = (-) \tens{\L} \omega_X[1].$$

It is well-known that $\omega_X$ is a divisorial sheaf (reflexive and generically locally free of rank 1) and in particular torsion-free. Indeed on the regular locus $U \subset X$, this definition reduces to give an isomorphism of coherent sheaves $\omega_X|_U \cong \det \Omega^1_{U/k}$, where the latter is the canonical bundle on $U$ (see for example, \cite[Proposition 22]{Kleiman-Relative_Duality} or \cite[Lemma 3.7.5]{Kovacs-Rational_singularities}). Further, the restriction of $\omega_X$ to the Gorenstein locus is a line bundle \cite[\href{https://stacks.math.columbia.edu/tag/0BFQ}{Lemma 0BFQ}]{stacks-project}. 

With this description of the dualizing sheaf $\omega_X$, let us introduce some convenient notation. Given a projective curve $X$, we define the \emph{non-Gorenstein locus} $X_{nG}$ to be the set consisting of all points $x \in X$ such that $\omega_{X,x}$ is not cyclic as a $\O_{X,x}$-module. By definition on $X \setminus X_{nG}$ the stalk $\omega_{X,x}$ of the canonical sheaf is generated by a single element and since it is torsion-free, it must be free of rank 1, hence $\omega_X|_U \cong \O_U$ for some open neighborhood containing $x$. Using this notation, we may reformulate our assumptions in Theorem \ref{Main Theorem} and say that a projective curve $X$ is strictly Cohen-Macaulay if and only if $X_{nG} \neq \emptyset$.

\subsection{The Proof of Theorem \ref{Main Theorem}}

The first step is to characterize the perfect zero-cycles as objects of $\perf X$. Recall from our conventions that all varieties are assumed integral. As with the definition of a perfect zero-cycle, the following is a direct generalization of \cite{SalasSalas-Reconstruction_from_the_derived_category}.

\begin{definition}\label{def-point}
  Let $P$ be an object of $\perf X$, $\iota,R_X: \perf X \to \D^b(X)$ the inclusion and Rouquier functor respectively. Then we say that $P$ is a strong point object of codimension $d\geq 0$ if it satisfies 
  \begin{enumerate}
    \item $R_X(P) \cong \iota(P)[d]$ 
    \item $\hom(P,P[j]) = 0$ for $j <0$,
    \item $\hom(P,P)$ is a commutative local $k$-algebra,
    \item $\hom(P,P[d])$, is a cyclic right $\hom(P,P)$-module.
  \end{enumerate}
  If $P$ only satisfies (2) through (4), we say that it is a weak point object (where the codimension is taken by default to be $\dim X=1$). 
\end{definition}

Its clear from the definition that the structure sheaf of a perfect zero-cycle supported in $X \setminus X_{nG}$ is a strong point object. However the perfect zero-cycles which are supported in $X_{nG}$ are weak point objects which are not strong;  see Lemma \ref{lemma_point support and codimension} below. In addition, we wish to point out that there may be weak point objects which are not structure sheaves of perfect zero-cycles, but see Lemma \ref{lemma_partial point} for a partial answer to this. 

\begin{lemma}\label{lemma_point support and codimension}
  Let $P$ be a strong point object of codimension $d$ on a projective curve $X$. Then the following hold:
  \begin{enumerate}
    \item $d = 1$, and
    \item $\operatorname{supp} P \cap X_{nG} = \emptyset$.
  \end{enumerate}
\end{lemma}

\begin{remark}
  Before we begin the proof, we remind the reader that a straightforward application of Nakayama's Lemma gives a convenient criterion for when a closed point $\iota_x:\{x\} \hookrightarrow X$ belongs to the support of a complex $C$, namely this occurs if and only if $\L \iota_x^* C \neq 0$.
\end{remark}

\begin{proof}
  Taking cohomology of both sides of the isomorphism $R_X(P) \cong i(P)[d]$ in Definition \ref{def-point} we see that 
  \begin{equation}\label{eqn_ point isomorphism}
    \mathcal{H}^{i+1} (P \tens{\L} \omega_X) \cong \mathcal{H}^{i+d}(P).
  \end{equation} 
  Now using that both complexes are bounded above, set $j = \max\{ i \,| \, \mathcal{H}^i(P) : = \mathcal{H}^i \neq 0\}$. It follows from right exactness of the tensor product that $$\mathcal{H}^j (P \tens{\L} \omega_X) \cong \mathcal{H}^j \tens{} \omega_X$$ and so $j$ is also maximal with respect to $\mathcal{H}^j(P\tens{\L}\omega_X) \neq 0$, as the right hand side vanishes if and only if $\mathcal{H}^j$ vanishes. Now in (\ref{eqn_ point isomorphism}), it follows that $i+1$ is maximal if and only if $i+d$ is maximal, and further that $i + 1 = i + d$. Hence $d=1$.

  Now fixing some closed point $\iota_x :\{x\}\hookrightarrow X$, passing to stalks via $\iota_x^{-1}$ is an exact functor, and so it commutes with taking cohomology of the complex $P$. Let $\mathcal{H}^i_x = \iota_x^{-1} \mathcal{H}^i = \mathcal{H}^i( \iota_x^{-1} (P))$ be the cohomology sheaves of the complex $P_x = \iota_x^{-1} (P)$. Then let $j_x$ be maximal such that $\mathcal{H}^{j_x}_x \neq 0$. Then again by right-exactness of the tensor product and exactness of the inverse image, taking right-most cohomology (which is in degree $j_x$ by assumption) of $\L \iota_x^* (P \tens{\L} \omega_X) \cong \L \iota_x^* P$ yields an isomorphism of $k(x)$-vector spaces $$\mathcal{H}^{j_x}_x \otimes_{\O_{X,x}} \omega_{X,x} \otimes_{\O_{X,x}} k(x) \cong \mathcal{H}^{j_x}_x \otimes_{\O_{X,x}} k(x).$$
  However if $x \in X_{nG}$ then $\omega_{X,x}$ is generated by at least two elements as a $\O_{X,x}$-module and by Nakayama's lemma the left hand side has strictly larger dimension then the right hand side. This is a contradiction, so we conclude that the cohomology sheaves of $P$ have support disjoint from $X_{nG}$. 
\end{proof}

\begin{lemma}\label{lemma_characterize point objects}
  Let $X$ be a projective curve which is strictly Cohen-Macaulay ($X_{nG} \neq \emptyset$). Then every point object $P$ in $\perf X$ is isomorphic to $\O_{Z_x}[r]$ for some $x \in X \setminus X_{nG}$ and $r \in \Z$, where $Z_x$ is a perfect zero-cycle supported at $x \in X$.
\end{lemma}
The following proof is almost completely identical to the analogous result in \cite{SalasSalas-Reconstruction_from_the_derived_category}. We include it here for the reader's convenience.
\begin{proof}
  Let $P = \O_{Z_x}$ be the structure sheaf of a perfect zero-cycle supported on some closed point $x \in X \setminus X_{nG}$. It follows from the fact that $x \in X \setminus X_{nG}$ that $P$ satisfies (1) in Definition \ref{def-point}. $P$ obviously satisfies (2) in Definition \ref{def-point}, as does any sheaf concentrated in some degree, and since $\O_{Z_x}$ is a cyclic $\O_X$-module, it follows that $$\hom_{\perf X}(P,P) \cong \hom_{\O_{X,x}}(\O_{Z_x},\O_{Z_x}) \cong \O_{Z_x}$$ and this is (3) in Definition \ref{def-point}. The last property is immediate from the definition of a perfect zero-cycle. This shows that perfect zero-cycles supported on $X \setminus X_{nG}$ are strong point objects.

  Conversely, the first step is to show that $P$ is supported at a single closed point. Lemma \ref{lemma_point support and codimension} implies that the support is proper. Since a closed proper subset of an integral curve is a finite set of points, we see $\dim (\operatorname{supp}P)=0$. If $P$ were supported on at least two distinct points, it follows that $P$ can be decomposed into a nontrivial direct sum. Yet the assumption that $\hom(P,P)$ be a local $k$-algebra implies that $P$ is indecomposable. All together this implies that $P$ is supported at a single point. 

  Next, we show that property (2) of Definition \ref{def-point} implies that $P$ has a single cohomology sheaf. Suppose that $P$ has two nontrivial cohomology sheaves, $\mathcal{H}^i$, $\mathcal{H}^j$, $i<j$ minimal and maximal respectively. Since we chose $i$ minimal, respectively $j$ maximal, we have a nonzero morphism $\mathcal{H}^i \to P[i]$, respectively $P[j] \to \mathcal{H}^j$. Now since $\mathcal{H}^j$ and $\mathcal{H}^i$ are supported at a single common point, Lemma \ref{lemma_local maps} implies the existence of a nonzero morphism $\mathcal{H}^j \to \mathcal{H}^i$. Hence the composition $$P[j] \to \mathcal{H}^j \to \mathcal{H}^i \to P[i]$$ is nontrivial, but this gives a nonzero element in $\hom(P[j],P[i]) = \hom(P,P[i-j])$ which is a contradiction as $i-j<0$. Hence $P \cong \F [r] \in \perf X$ for some $r \in \Z$ and $\F \in \coh X$.

  Now we have an exact sequence $$ 0 \to \mathcal{G} \to \F \to \F/ \mathfrak{m}_x \F \to 0$$ in $\coh X$, where $\mathcal{G}$ is the kernel of the projection $\F \to \F/\mathfrak{m}_x \F \cong \F \otimes_{\O_{X,x}} k(x)$. This gives a distinguished triangle in $\D^b(X)$, and applying $\hom(\F,-)$ gives a long exact sequence of abelian groups. Now since $\F$ is a coherent sheaf with zero-dimensional support, it has $\operatorname{depth} \F =0$ as a module over the local ring. The Auslander-Buchsbaum formula implies, since the local ring is Cohen-Macaulay, that the projective dimension of $\F$ is $\dim X=1$.
  
  Thus, in the long exact sequence of abelian groups, we have a surjection $$\cdots \to \hom(\F,\F[1]) \to \hom(\F, \F/ \mathfrak{m}_x \F [1]) \to 0,$$ and since the former is a cyclic $\hom(\F,\F)$-module, the latter must be as well. In particular this implies that $\F/\mathfrak{m}_x \F$ is a one dimensional $k(x)$-vector space, as otherwise, $\hom(\F , \F/\mathfrak{m}_x \F [1])$ would split into a direct sum of Hom spaces, and hence not be cyclic. This then implies that $\F$ is cyclic by Nakayama's lemma, and so $\F \cong \O_X/ \mathcal{I}$ with $\mathcal{I}$ the annihilator of $\F$ in $\O_X$ which is $\mathfrak{m}_x$-primary. Taking the closed subscheme $Z_x$ defined by $\mathcal{I}$, we see that $Z_x$ is supported at a single closed point $x \in X\setminus X_{nG}$, and $P \cong \O_{Z_x}[r]$. Hence the claim. 
\end{proof}

The same method of proof also yields the following.
\begin{lemma}\label{lemma_partial point}
  Let $X$ be projective curve and $P \in \perf X$ be a weak point object with zero-dimensional support. Then $P \cong \O_{Z_x}[r]$ for some $x\in X$ and $r\in \Z$.
\end{lemma}

As presented in Lemma \ref{LuntsOrlov_FM for perf}, any equivalence $\Phi:\perf X \to \perf Y$ is isomorphic to a Fourier-Mukai functor $$\Phi^{X \to Y}_\K(-) = \R \pi_{Y*}(\pi_X^*(-) \tens{\L} \mathcal{K}),$$ where $$X \overset{\pi_X}{\leftarrow} X \times Y \overset{\pi_Y}{\rightarrow} Y$$ are the projections and $\K \in \D^b(X \times Y)$. Here and from now on, fix one such Fourier-Mukai functor. Using the perfect zero-cycles, we aim to show that (extending $\Phi^{X \to Y}_\K$ to an equivalence $\D^b(X)\cong \D^b(Y)$) the functor preserves the structure sheaves of closed points. 

\begin{remark}\label{rem_FM preserves point objects}
  From the definition of a point object it is clear that an equivalence $\perf X \cong \perf Y$ must preserve the weak point objects, however since the Rouquier functor is not an endofunctor of $\perf X$, one may worry that the strong point objects are not preserved. Recall that from Lemma \ref{lemma_equivalences}, the Fourier-Mukai functor $\Phi^{X \to Y}_\K:\perf X \cong \perf Y$ extends to an equivalence $\D^b(X) \cong \D^b(Y)$, and since the Rouquier functor is valued in the bounded derived category of coherent sheaves by Lemma \ref{lemma_Rouquier functor is dualizing complex} and commutes with equivalences as in Lemma \ref{lemma_Rouquier commutes with equivalences}, we get a commutative square
  
  \begin{center}
    \begin{tikzcd}
      \operatorname{Perf} X \arrow[r, "\Phi^{X \to Y}_\K"] \arrow[d, "R_X"'] & \operatorname{Perf} Y \arrow[d, "R_Y"] \\
      \mathrm{D}^b(X) \arrow[r, "\Phi^{X \to Y}_\K"]                         & \mathrm{D}^b(Y).                       
      \end{tikzcd}
  \end{center}
  
  Now from Definition \ref{def-point}, strong point objects remain perfect complexes after an application of the Rouquier functor, so by this observation and the commutativity of the above, it follows that if $P$ is a strong point object in $\perf X$, then so is $\Phi^{X \to Y}_\K(P)$ in $\perf Y$.
  
  Note that this immediately implies that $\dim X = \dim Y=1$ by considering the codimensions of strong points objects in both categories. Since then $Y$ is assumed to be integral, $\dim Y=1$, and so $Y$ is Cohen-Macaulay for dimension reasons.
\end{remark}

We now see that the assumption that we have an equivalence implies that although $Y$ was not assumed to be strictly Cohen-Macaulay, all strong point objects are shifts of perfect zero-cycles.

\begin{lemma}\label{lemma_FM preserves points}
  Let $X$ and $Y$ be projective curves with $X$ strictly Cohen-Macaulay, and $\Phi^{X \to Y}_\K:\perf X \to \perf Y$ an equivalence of categories. Then all strong point objects on $Y$ are of the form $\O_{Z_y}[r]$ for a perfect zero-cycle $Z_y$ and $r\in \Z$. Further, given any closed point $x\in X$, there exists a closed point $y\in Y$ and $r\in \Z$ such that $\Phi^{X \to Y}_\K(\O_{Z_x}) \cong \O_{Z_y}[r]$.
\end{lemma}
\begin{proof}
  We first show that the existence of the equivalence $\Phi^{X \to Y}_\K$ implies that every strong point object of $\perf Y$ is of the same form. Denote by $\tilde{P}(-)$ the strong point objects in $\perf(-)$, and by $P(-)$ the objects in $\perf (-)$ isomorphic to a shift of a structure sheaf of a perfect zero-cycle supported off the non-Gorenstein locus. Then by definition of the objects in $\tilde{P}(X)$ and Remark \ref{rem_FM preserves point objects}, they are preserved by equivalences, and so Lemma \ref{lemma_characterize point objects} and that $\Phi^{X \to Y}_\K$ is an equivalence gives us the following inclusions of sets $$P(X) = \tilde{P}(X) \cong \tilde{P}(Y) \supset P(Y).$$ 
  We would like to show that $\tilde{P}(Y) = P(Y)$. Consider a strong point object $Q \in \tilde{P}(Y)$. By Lemma \ref{lemma_partial point}, it is enough to show that this strong point object has zero dimensional support. 
  
  Note that by the inclusion of sets above, $Q$ is in the essential image of the Fourier-Mukai functor, that is $\Phi^{X \to Y}_\K(\O_{Z_{x'}}[r]) \cong Q$ for some closed point $x' \in X$, $r\in \Z$, and perfect zero cycle $Z_{x'}$. Similarly, given any strong point object $\O_{Z_y}$ supported on a closed point $y \in Y \setminus Y_{nG}$, we can find $x \in X \setminus X_{nG}$ such that $\Phi^{X \to Y}_\K(\O_{Z_{x}}) = \O_{Z_{y}}[s]$ for some $s\in \Z$. By choosing some other closed point in $Y \setminus Y_{nG}$ if necessary, we can arrange that $x \neq x'$.  Since $\Phi^{X \to Y}_\K$ is an equivalence of categories, and strong point objects with disjoint support are homologically orthogonal, it follows from fully faithfulness that $$\hom(\O_{Z_{x'}},\O_{Z_x}) \cong \hom(Q,\O_{Z_{y}}[t]) = 0 $$ for all $t\in \Z$, but by Lemma \ref{lemma_perfect point test support} this means $Q$ is supported on a proper (zero-dimensional) subset. Hence by Lemma \ref{lemma_partial point} $Q \cong \O_{Z_x}[r]$ for some $x \in X \setminus X_{nG}$ and $r \in \Z$. This shows the first statement.

  Now we want to show that perfect zero-cycles (regardless of where they are supported) are preserved by the Fourier-Mukai functor. If they are supported in $X \setminus X_{nG}$, the previous paragraph implies the claim, so let $x \in X_{nG}$. Again consider the complex $Q = \Phi^{X \to Y}_\K(\O_{Z_x})$. Choosing some $x' \in X \setminus X_{nG}$, the conclusion of the previous paragraph says $\Phi^{X \to Y}_\K(\O_{Z_{x'}}) \cong \O_{Z_{y'}}[r]$ for some $r \in \Z$. Fully faithfulness implies that $$\hom(\O_{Z_x},\O_{Z_{x'}}) \cong \hom(Q,\O_{Z_{y'}}[r]) =0,$$ for all $r\in \Z$, which by Lemma \ref{lemma_perfect point test support} tells us that $\operatorname{supp}Q$ is a proper subset of $Y$, and being closed, must be zero-dimensional since $Y$ is an integral curve. Now again Lemma \ref{lemma_partial point} implies that $Q \cong \O_{Z_y}[s]$ for some $y \in Y_{nG}$ and $s \in \Z$. This completes the proof.
\end{proof}

\begin{remark}\label{rem_FM preserves smooth points}
  Note that Lemma \ref{lemma_FM preserves points}, together with Remark \ref{rem_FM preserves point objects}, implies for all closed points $x \in X$, there is a closed point $y \in Y$ such that $F(\O_{Z_x}) = \O_{Z_y}[r]$ for some $r\in \Z$. It follows form this statement that if $x \in X$ is a regular point, so is $y \in Y$. Indeed by \ref{remark_regular iff perfect} it is enough to see that if $\O_{Z_x} \cong k(x)$, then $\O_{Z_y} \cong k(y)$. Since $\Phi^{X \to Y}_\K$ is fully faithful, $$\hom(k(x),k(x)) \cong \hom(\O_{Z_y},\O_{Z_y}).$$ Since the left hand side is the residue field $k(x)$, so must the right hand side. But the only module over $\O_{Y,y}$ with this endomorphism algebra is the residue field. This shows $\O_{Z_y} \cong k(y)$ and in particular, $y \in Y$ is a smooth point.
\end{remark}

We are now in a position to prove Theorem \ref{Main Theorem}.
 
\begin{proof}[Proof of Theorem \ref{Main Theorem}]
  We show that the Fourier-Mukai functor $\Phi^{X \to Y}_\K$ preserves the structure sheaves of closed points, and rest of the proof follows from Lemma \ref{lemma_skyscraper sheaves give iso}.
  
  Using Lemma \ref{lemma_equivalences}, first extend the equivalence $\Phi^{X \to Y}_\K$ to an equivalence $\D^b(X) \overset{\sim}{\to} \D^b(Y)$.  Fix an arbitrary closed point $x\in X$, and denote by $Q$ the complex $\Phi^{X \to Y}_\K(k(x))$. If we consider any other closed point $x' \in X$ which is distinct from $x$, it follows via Lemma \ref{lemma_perfect point test support} that $\hom_{\D^b(X)}(k(x),\O_{Z_{x'}}[r])=0$ for all $r\in \Z$. Since $\Phi^{X \to Y}_\K$ is fully faithful, $$\hom_{\D^b(X)}(k(x),\O_{Z_{x'}}[r]) \cong \hom_{\D^b(Y)}(Q,\Phi^{X \to Y}_\K(\O_{Z_{x'}})[r]) = 0$$ for all $r\in \Z$ as well. By Lemma \ref{lemma_FM preserves points}, the Fourier-Mukai preserves the perfect zero-cycles, and so there exists some $y' \in Y$ and $s \in \Z$ such that $\Phi^{X \to Y}_\K(\O_{Z_{x'}})[r] \cong \O_{Z_{y'}}[s]$. Finally, applying once again Lemma \ref{lemma_perfect point test support}, it follows that $\operatorname{supp}Q$ is a proper closed subset of $Y$, hence zero-dimensional.

  The rest of the proof proceeds similarly to the proof of Lemma \ref{lemma_characterize point objects}. Noting that $$\hom_{\D^b(X)}(k(x),k(x)[i]) = \begin{cases}
    0 & i<0 \\ k(x) & i=0,
  \end{cases}$$ 
  and once again using that $\Phi^{X \to Y}_\K$ is fully faithful, it follows that $$\hom(Q, Q[i]) = \begin{cases}
    0 & i<0 \\ k(x) & i=0.
  \end{cases}$$ 
  First consider when $i=0$. Since the endomorphisms are a field, the complex $Q$ must be supported at a single point. Now consider when $i<0$. Then since $Q$ is supported at a single point, an identical argument as in Lemma \ref{lemma_characterize point objects} implies $Q \cong \F [r]$ for some $\F \in \coh Y$, supported at a single point $y \in Y$, and $r \in \Z$. Considering $i=0$ once again, $\operatorname{End}(Q) \cong k(x)$ implies that $Q \cong k(y)[r]$ for some $r \in \Z$ and some $y\in Y$ a closed point. This is what we wanted to show.
\end{proof}

\bibliographystyle{alpha} 
\bibliography{references}

\begin{thebibliography}{HRLMSdS09}

\bibitem[Bal05]{Balmer-Spectrum_of_prime_ideals_in_tensor_categories}
Paul Balmer.
\newblock The spectrum of prime ideals in tensor triangulated categories.
\newblock {\em J. Reine Angew. Math.}, 588:149--168, 2005.

\bibitem[Bal11]{Ballard-Derived_categories_of_singular_schemes_and_reconstruction}
Matthew Ballard.
\newblock Derived categories of sheaves on singular schemes with an application
  to reconstruction.
\newblock {\em Adv. Math.}, 227(2):895--919, 2011.

\bibitem[BBHR09]{Bartocci-FM_in_geometry_and_physics}
Claudio Bartocci, Ugo Bruzzo, and Daniel Hern\'{a}ndez~Ruip\'{e}rez.
\newblock {\em Fourier-{M}ukai and {N}ahm transforms in geometry and
  mathematical physics}, volume 276 of {\em Progress in Mathematics}.
\newblock Birkh\"{a}user Boston, Inc., Boston, MA, 2009.

\bibitem[Ber07]{Bernardara-FM_curves}
Marcello Bernardara.
\newblock Fourier-{M}ukai transforms of curves and principal polarizations.
\newblock {\em C. R. Math. Acad. Sci. Paris}, 345(4):203--208, 2007.

\bibitem[BH93]{BrunsHerzog-CM_rings}
Winfried Bruns and J\"urgen Herzog.
\newblock {\em Cohen-Macaulay rings}, volume~39 of {\em Cambridge Studies in
  Advanced Mathematics}.
\newblock Cambridge University Press, Cambridge, 1993.

\bibitem[BO01]{BondalOrlov-Reconstruction_from_the_derived_category}
Alexei Bondal and Dmitri Orlov.
\newblock Reconstruction of a variety from the derived category and groups of
  autoequivalences.
\newblock {\em Compositio Math.}, 125(3):327--344, 2001.

\bibitem[BvdB03]{BondalVandenBergh-Generators_and_representability}
A.~Bondal and M.~van~den Bergh.
\newblock Generators and representability of functors in commutative and
  noncommutative geometry.
\newblock {\em Mosc. Math. J.}, 3(1):1--36, 258, 2003.

\bibitem[Cat82]{Catanese-Pluricanonical_gorenstein_curves}
Fabrizio Catanese.
\newblock Pluricanonical {G}orenstein curves.
\newblock In {\em Enumerative geometry and classical algebraic geometry
  ({N}ice, 1981)}, volume~24 of {\em Progr. Math.}, pages 51--95.
  Birkh\"{a}user Boston, Boston, MA, 1982.

\bibitem[CS14]{CanonacoStellari-FM_supported}
Alberto Canonaco and Paolo Stellari.
\newblock Fourier-{M}ukai functors in the supported case.
\newblock {\em Compos. Math.}, 150(8):1349--1383, 2014.

\bibitem[CS18]{CanonacoStellari-Uniqueness_of_DG_enhancements}
Alberto Canonaco and Paolo Stellari.
\newblock Uniqueness of dg enhancements for the derived category of a
  {G}rothendieck category.
\newblock {\em J. Eur. Math. Soc. (JEMS)}, 20(11):2607--2641, 2018.

\bibitem[Fav12]{Favero-Reconstruction_and_finiteness}
David Favero.
\newblock Reconstruction and finiteness results for {F}ourier-{M}ukai partners.
\newblock {\em Adv. Math.}, 230(4-6):1955--1971, 2012.

\bibitem[HRLMdS07]{Ruiperez-Fourier-Mukai_transforms_for_Gorenstein_schemes}
Daniel Hern\'{a}ndez~Ruip\'{e}rez, Ana~Cristina L\'{o}pez~Mart\'{\i}n, and
  Fernando~Sancho de~Salas.
\newblock Fourier-{M}ukai transforms for {G}orenstein schemes.
\newblock {\em Adv. Math.}, 211(2):594--620, 2007.

\bibitem[HRLMSdS09]{Ruiperez-Relative_integral_functor_for_singular_fibrations_and_partners}
Daniel Hern\'{a}ndez~Ruip\'{e}rez, Ana~Cristina L\'{o}pez~Mart\'{\i}n, and
  Fernando Sancho~de Salas.
\newblock Relative integral functors for singular fibrations and singular
  partners.
\newblock {\em J. Eur. Math. Soc. (JEMS)}, 11(3):597--625, 2009.

\bibitem[Huy06]{Huybrechts-FM_transforms_in_algebraic_geometry}
D.~Huybrechts.
\newblock {\em Fourier-{M}ukai transforms in algebraic geometry}.
\newblock Oxford Mathematical Monographs. The Clarendon Press, Oxford
  University Press, Oxford, 2006.

\bibitem[HVdB07]{HilleVandenBergh-FM_transforms}
Lutz Hille and Michel Van~den Bergh.
\newblock Fourier-{M}ukai transforms.
\newblock In {\em Handbook of tilting theory}, volume 332 of {\em London Math.
  Soc. Lecture Note Ser.}, pages 147--177. Cambridge Univ. Press, Cambridge,
  2007.

\bibitem[KKS18]{Karmazyn-Derived_categories_of_singular_surfaces}
Joseph {Karmazyn}, Alexander {Kuznetsov}, and Evgeny {Shinder}.
\newblock Derived categories of singular surfaces.
\newblock arXiv:1809.10628, 2018.

\bibitem[Kle80]{Kleiman-Relative_Duality}
Steven~L. Kleiman.
\newblock Relative duality for quasicoherent sheaves.
\newblock {\em Compositio Math.}, 41(1):39--60, 1980.

\bibitem[Kov17]{Kovacs-Rational_singularities}
Sándor~J Kovács.
\newblock Rational singularities.
\newblock arXiv:1703.02269, 2017.

\bibitem[LM14]{Martin-Fourier-mukai_partners_of_singular_genus_one_curves}
Ana~Cristina L\'{o}pez~Mart\'{\i}n.
\newblock Fourier-{M}ukai partners of singular genus one curves.
\newblock {\em J. Geom. Phys.}, 83:36--42, 2014.

\bibitem[LMTP17]{MartinPrieto-Derived_equivalences_and_kodaira_fibers}
Ana~Cristina L\'{o}pez~Mart\'{\i}n and Carlos Tejero~Prieto.
\newblock Derived equivalences and {K}odaira fibers.
\newblock {\em J. Geom. Phys.}, 122:69--79, 2017.

\bibitem[LO10]{LuntsOrlov-Uniqueness_of_DG_enhancements}
Valery~A. Lunts and Dmitri~O. Orlov.
\newblock Uniqueness of enhancement for triangulated categories.
\newblock {\em J. Amer. Math. Soc.}, 23(3):853--908, 2010.

\bibitem[Muk81]{Mukai-Duality_and_picard_sheaves}
Shigeru Mukai.
\newblock Duality between {$\D(X)$} and {$\D(\hat X)$} with its application to
  picard sheaves.
\newblock {\em Nagoya Math. J.}, 81:153--175, 1981.

\bibitem[Nee96]{Neeman-Grothendieck_duality_via_Brown_representability}
Amnon Neeman.
\newblock The {G}rothendieck duality theorem via {B}ousfield's techniques and
  {B}rown representability.
\newblock {\em J. Amer. Math. Soc.}, 9(1):205--236, 1996.

\bibitem[Orl97]{Orlov-Equivalences_and_K3_surfaces}
D.~O. Orlov.
\newblock Equivalences of derived categories and {$K3$} surfaces.
\newblock volume~84, pages 1361--1381. 1997.
\newblock Algebraic geometry, 7.

\bibitem[Orl02]{Orlov-Equivalences_of_abelian_varieties}
D.~O. Orlov.
\newblock Derived categories of coherent sheaves on abelian varieties and
  equivalences between them.
\newblock {\em Izv. Ross. Akad. Nauk Ser. Mat.}, 66(3):131--158, 2002.

\bibitem[Riz17]{Rizzardo-Adjoints_for_FM}
Alice Rizzardo.
\newblock Adjoints to a {F}ourier-{M}ukai functor.
\newblock {\em Adv. Math.}, 322:83--96, 2017.

\bibitem[Rou08]{Rouquier-Dimensions_of_triangulated_categories}
Rapha\"{e}l Rouquier.
\newblock Dimensions of triangulated categories.
\newblock {\em J. K-Theory}, 1(2):193--256, 2008.

\bibitem[SdSSdS12]{SalasSalas-Reconstruction_from_the_derived_category}
Carlos Sancho~de Salas and Fernando Sancho~de Salas.
\newblock Reconstructing schemes from the derived category.
\newblock {\em Proc. Edinb. Math. Soc. (2)}, 55(3):781--796, 2012.

\bibitem[{Sta}19]{stacks-project}
The {Stacks project authors}.
\newblock The stacks project.
\newblock \url{https://stacks.math.columbia.edu}, 2019.

\bibitem[TT90]{ThomasonTrobaugh-Higher_K_theory}
R.~W. Thomason and Thomas Trobaugh.
\newblock Higher algebraic {$K$}-theory of schemes and of derived categories.
\newblock In {\em The {G}rothendieck {F}estschrift, {V}ol. {III}}, volume~88 of
  {\em Progr. Math.}, pages 247--435. Birkh\"{a}user Boston, Boston, MA, 1990.

\bibitem[War14]{Ward-Arithemtic_Calabi_Yau}
Matthew Ward.
\newblock {\em Arithmetic {P}roperties of the {D}erived {C}ategory for
  {C}alabi-{Y}au {V}arieties}.
\newblock ProQuest LLC, Ann Arbor, MI, 2014.
\newblock Thesis (Ph.D.)--University of Washington.

\end{thebibliography}

\end{document}